\documentclass[12pt,a4paper]{article}
\usepackage[centertags]{amsmath}
\usepackage{amsfonts}
\usepackage{graphics}
\usepackage{graphicx}
\usepackage{amssymb}
\usepackage{amsthm}
\usepackage{listings}

\theoremstyle{plain}
\newtheorem{thm}{Theorem}[section]

\newtheorem{lem}[thm]{Lemma}

\theoremstyle{definition}
\newtheorem{defn}{Definition}[section]
\newtheorem{obsvn}{Observation}[section]

\theoremstyle{remark} \tolerance=10000 \hbadness=10000
\def \ni{\noindent}

\title{\bf Sudoku Number of Corona of Graphs}
\author{Manju S Nair \footnote{E-mail : manjunsnair@gmail.com}\\
	Department of General Sciences\\
	Govt. Women's Polytechnic College\\
	Ernakulam, Kerala, India. \and Aparna Lakshmanan S.\footnote{E-mail : aparnals@cusat.ac.in, aparnaren@gmail.com}\\
	Department of Mathematics\\
	Cochin University of Science and Technology\\
	Cochin - 22, Kerala, India. \and S Arumugam \footnote{E-mail : s.arumugam.klu@gmail.com}\\
	Department of Mathematics\\
	M S University\\
	Tirunelveli, Tamilnadu,India.
}
\date{}
\begin{document}
\maketitle
\begin{abstract}
\ni\line(1,0){360}\\
Let $ G = (V,E) $ be a graph of order $n$ with chromatic number $ \chi(G) = k$ and let $ S \subset V $ and let $ C_0 $ be a $k$-coloring of the induced subgraph $ G[S] $. The coloring $C_0$ is called an extendable coloring, if $C_0$ can be extended to a $k$-coloring of $G$ and it is a Sudoku coloring of $G$ if the extension is unique. The smallest order of such an induced subgraph $G[S]$ of $G$ which admits a Sudoku coloring is called the Sudoku number of $G$ and is denoted by $sn(G)$. In this paper, we first introduce the notion of uniquely color extendable vertex and then we obtain the lower and upper bounds for the Sudoku number of $G \circ H$ where $H$ is any graph of order $m \leq \chi(G)-2$. Some families of graphs which attain these bounds are also obtained. The exact value of the Sudoku number of $C_n \circ K_1$, $W_n \circ K_n$ for $n=1,2$, $K_n \circ K_m$ and $C_n \circ P_m$ for every $n,m \in N$ are also obtained.
\ni\line(1,0){360}\\
\ni {\bf Keywords:}  Sudoku Number, Corona product\\

\ni {\bf AMS Subject Classification:} {\bf Primary: 05C15, Secondary: 05C76}\\
\ni\line(1,0){360}\\
\end{abstract}

\section{Introduction}
By a graph $ G = (V,E) $, we mean a finite undirected connected simple graph of order $ n(G) = |V| $ and size $m(G) = |E| $ (if there is no ambiguity in the choice of $ G $, then we write it as $n$ and $m$, respectively). A vertex coloring of $ G $ is a map $ f:V \rightarrow C $, where $C$ is a set of distinct colors. It is proper if adjacent vertices of $G$ receive distinct colors of $C$, that is, if $ uv \in E(G) $, then $ f(u) \neq f(v) $. The minimum number of colors needed for a proper vertex coloring of $G$ is called the chromatic number $ \chi(G) $, of a graph $G$.\\

Sudoku coloring is a concept which is motivated by the popular Sudoku puzzle. Let $V$ denote the 81 cells consisting of 9 rows, 9 columns and nine 3 x 3 subsquares. Let $G$ be the graph with vertex set $V$ in which each of the rows, columns and the nine 3 x 3 subsquares are complete graphs. Clearly $ \chi(G) = 9 $ and the solution of a Sudoku puzzle gives a proper vertex coloring of $G$ with 9 colors. A Sudoku puzzle corresponds to a proper vertex coloring $C$ of an induced subgraph $H$ of $G$ using atmost 9 colors with the property that $C$ can be uniquely extended to a 9-coloring of $G$. This motivates the definition of Sudoku coloring \cite{maria2023sudoku}. \\

Let $ \chi(G) = k $ and $S \subset V$. Let $ C_0 $ be a $k$-coloring of the induced subgraph $ G[S] $. The coloring $C_0$ is called an extendable coloring, if $C_0$ can be extended to a $k$-coloring of $G$. $C_0$ is a Sudoku coloring of $G$, if this extension is unique. The smallest order of such an induced subgraph $G[S]$ of $G$ which admits a Sudoku coloring is called the Sudoku number of $G$ and is denoted by $sn(G)$\cite{maria2023sudoku}.\\

The corona product of two graphs $G$ and $H$, denoted by $G \circ H$, is defined as the graph obtained by taking one copy of $G$ and $|V(G)| $ copies of $H$ and joining the $i^{\text{th}}$ vertex of $G$ to every vertex in the $i^{\text{th}}$ copy of $H$. \\

Gee-Choon Lau et al., introduced the concept of Sudoku number in \cite{maria2023sudoku}. They showed that this parameter is related to the list coloring of graphs and they obtained necessary conditions for $C_0$ being a Sudoku coloring. They have also determined the Sudoku number of various families of graphs.
Alexey Pokrovskiey, proved that a connected graph has Sudoku number $n-1$, if and only if, it is complete \cite{pokrovskiy2022graphs}. The same concept has been considered by Cooper and Kirkpatrick \cite{cooper2014critical} and J A Bate and G H J Van Rees \cite{bate1999size} using critical sets. Also it has been independently defined by E S Mahmoodian et al. \cite{mahmoodian1997defining} under the name defining sets. But, their proof techniques and approach to the concept were different.\\

In this paper, we first introduce the notion of uniquely color extendable vertex. Then we obtain the lower and upper bounds for the Sudoku number of $G \circ H$ where $H$ is any graph of order $m \leq \chi(G)-2$ and some families of graphs which attain the bounds. Also, we determine the Sudoku number of $C_n \circ K_1$, $W_n \circ K_n$ for $n=1,2$, $K_n \circ K_m$ and $C_n \circ P_m$ for every $n,m \in N$.\\

For notations and concepts not mentioned here, we refer to \cite{balakrishnan2012textbook}.\\

\begin{defn}
Let $C_0$ be an extendable coloring of $G[S]$ for $S \subset V(G)$. A vertex $u \in V - S$ is called a uniquely color extendable (u.c.e) vertex if the color that can be given to $u$ in any proper extension of $C_0$ is unique.
\end{defn}

\begin{obsvn} An extendable coloring of $G[S]$ for $S \subset V(G)$ is a Sudoku coloring, if every vertex in $V - S$ is iteratively a u.c.e vertex. Also, if a vertex $w$ is a u.c.e. vertex then one of the following holds.
\begin{itemize}
  \item[1)] $w$ is a color dominating vertex (ie; $N(w)$ contains all colors of $C_0$, except one color)
  \item[2)] $N(w)$ has $k$ u.c.e. vertices $\{v_1, v_2, \dots, v_k\}$ such that $v_i$ attracts color $i$ for every $i$ and $N(w) \cap V_j \neq \phi$ for exactly $\chi(G)- k - 1$ values of $j \in \{k+1, k+2, \dots, \chi(G)\}$.
  \item[3)] $N(w)$ has an induced subgraph of chromatic number $\chi(G) – 1$ and there exists an $i \in C$ such that the color list of each vertex of this subgraph is a subset of $C – \{i\}$.
\end{itemize}
\end{obsvn}
\begin{obsvn}
	Let $G$ be a graph with $ \chi(G) \geq 3 $. Suppose $C_0$ is an extendable coloring of $G[S]$ for $S \subset V(G)$. If there is a vertex $v \notin S$ such that $deg(v) \leq k-2$, then $C_0$ is not a Sudoku coloring of $G$.
\end{obsvn}

The following lemmas and definition are useful for us.
\begin{lem}
Let $G$ be a graph with $\chi(G) \geq 3$. Suppose $C_0$ is an extendable coloring of $G[S]$ for $S \subset V(G)$. If there is a pendant vertex $v \notin S$, then $C_0$ is not a Sudoku coloring \cite{maria2023sudoku}.
\end{lem}

\begin{defn}
  Given a set $L(v)$ of colors for each vertex $v$ (called a list), a list coloring is a choice function that maps every vertex $v$ to a color in the list $L(v)$ \cite{jensen2011graph}.
\end{defn}

\begin{lem}
Suppose there exists a list coloring of the cycle $C_n = v_1v_2 \dots v_nv_1$ such that the list of colors for each vertex $v_i$ satisfies the following conditions.
\begin{itemize}
  \item $|L(v_i)| \geq 2$
  \item $L(v_i) \subseteq \{1,2,3\}$
\end{itemize} Then there are at least two list colorings of $C_n$ \cite{maria2023sudoku}.
\end{lem}
\begin{lem}
Let $L(x_i)$ be a list of colors of a vertex $x_i$ in the path $P_n = x_1x_2\dots x_n$, $1 \leq i \leq n$. If $|L(x_i)| \geq 2$ for each $i$, then there are at least 2 list colorings of $P_n$ \cite{maria2023sudoku}.
\end{lem}
\begin{lem}
	Let $G$ be a graph with $\chi(G) \geq 3$. Suppose $C_0$ is an extendable coloring of $G[S]$ for $S \subset V(G)$. If there is an edge $xy$ for which $x,y \notin S$ such that $deg(x) \leq \chi(G)-1 $ and $deg(y) \leq \chi(G)-1$, then $C_0$ is not a Sudoku coloring of $G$\cite{maria2023sudoku}.
\end{lem}
\section{Corona Product of Graphs}

\begin{thm}
Let $G$ be a graph with $\chi(G) \geq 3$. Then,
\begin{equation*}
  nm \leq sn(G \circ H) \leq nm + sn(G).
\end{equation*}
where $H$ is any graph of order $m \leq \chi(G)-2$.
\label{th1}
\end{thm}
\begin{proof}
Given that $\chi(G) \geq 3$. By Observation 1.2, all vertices in each copy of $H$ must be initially colored in any Sudoku coloring. Therefore, $sn(G \circ H) \geq nm$. Since $sn(G)$ is the Sudoku number of $G$, a Sudoku coloring of $G$ together with suitable colors given to the vertices in each copy of $H$ gives a uniquely extendable coloring of $G \circ H$ with $nm + sn(G)$ colors. Hence, $sn(G \circ H) \leq nm + sn(G)$. Hence the result.
\end{proof}
The bounds in the above theorem are tight and we have infinitely many graphs which attains the lower bound and which attains the upper bound.\\

\ni {\bf Graphs which attains the upper bound:} Let $G$ be a graph with $n(G)$ vertices and $\chi(G) \geq 4$. Let $G'$ be the graph obtained from $G$ by adding $ \chi(G) - 1$ pendant vertices to each vertex $v$ of $G$. Then $ n(G') = n(G)\chi(G) $.
By Lemma 1.1, all the $ n(G)[\chi(G)-1] $ pendant vertices of $ G' $, must be initially colored in any of the Sudoku coloring of $ G' $.
Therefore $ sn(G') \geq n(G)[\chi(G)-1] $. Let $C$ be an initial coloring of $G'$ as follows:

Color the $ \chi(G)-1 $ pendant vertices of each vertex $v$ of $G$ with $ \chi(G)-1 $ colors other than the color given to $v$ in the proper $ \chi- $ coloring of $G$. This will lead to a unique $ \chi- $ coloring of $ G' $. Hence $ sn(G') = n(G)[\chi(G)-1]$.\\

Now in Theorem \ref{th1}, take $G=G'$ and $H=K_1$ so that $nm+sn(G)= n(G')+sn(G')$. By Lemma 1.1, all the $ n(G)\chi(G) $ pendant vertices of $ G' \circ K_1 $ must be initially colored in any Sudoku coloring of $ G' \circ K_1 $. Also, the degree of each of the $ n(G)[\chi(G)-1]$ pendant vertices of $ G' $ is 2 in $ G' \circ K_1 $. Since $ \chi(G' \circ K_1) \geq 4 $, the color list of each of those vertices will be greater than or equal to 2. Hence, if any of the above $ n(G)[\chi(G)-1]$ vertices is not initially colored, it will not lead to a unique coloring of $ G' \circ K_1 $. Therefore,
\begin{equation}
	sn(G' \circ K_1) \geq n(G)\chi(G) + n(G)[\chi(G)-1]	
	\label{eqn1}
\end{equation}

Since $ n(G)[\chi(G)-1]$ is the Sudoku number of $ G' $, a Sudoku coloring of $ G' $ together with the suitable colors given to the $ n(G)\chi(G)$ pendant vertices of $ G' \circ K_1 $, gives a uniquely extendable coloring of $ G' \circ K_1 $. Therefore,
\begin{equation}
	sn(G' \circ K_1) \leq n(G)\chi(G) + n(G)[\chi(G)-1]
	\label{eqn2}
\end{equation}

From (\ref{eqn1}) and (\ref{eqn2}),
\begin{equation*}
	sn(G' \circ K_1) = n(G)\chi(G) + n(G)[\chi(G)-1] = n(G') + sn(G').
\end{equation*}

\ni{ \bf Graphs which attains the lower bound:} Now to prove that the lower bound is sharp, consider $G=$Line graph of $ C_n \circ K_1$ and $H=K_1$ in Theorem \ref{th1} so that $n(G)m=2n$. Figure \ref{f:L(CnoK1)} shows the graph $G$.

\begin{figure}[h]
	\centering
	\includegraphics[width=8cm]{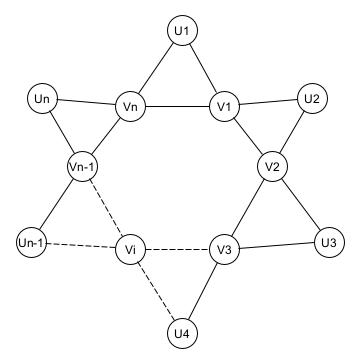}
	\caption{Line graph of $ C_n \circ K_1 $}
	\label{f:L(CnoK1)}
\end{figure}

 We know $ \chi(G \circ K_1) = 3 $. Let $ u_j' $ and $v_j' $ be the pendant vertices attached to $ u_j $ and $ v_j $ respectively, $ j = 1,2,3, \dots , n $. By Lemma 1.1, all the $ 2n $ pendant vertices of $ G \circ K_1 $ must be initially colored in any Sudoku coloring of $ G \circ K_1 $. Therefore $ sn(G \circ K_1)  \geq 2n $. \\

Now let $ n = 3k $ or $ 3k+2 $ ,where k is any natural number. Then the following is a uniquely extendable coloring of $ G \circ K_1 $.

\begin{equation*}
	 C(u_j')  =  C(v_j')  = \begin{cases}
		1, & \mbox{if } j \equiv 1(\text{mod} 3) \\
		2, & \mbox{if } j \equiv 2(\text{mod} 3) \\
		3, & \mbox{if } j \equiv 0(\text{mod} 3).
	\end{cases}
\end{equation*}

Again if $ n = 3k+1 $ , where k is any natural number, then the following is a uniquely extendable coloring of $ G \circ K_1 $.

\begin{equation*}
	C(u_n') = C(v_n') = 2 \and \text{ and } C(u_j') = C(v_j') = \begin{cases}
		1, & \mbox{if } j \equiv 1(\text{mod} 3) \text{ and } j \neq n \\
		2, & \mbox{if } j \equiv 2(\text{mod} 3) \\
		3, & \mbox{if } j \equiv 0(\text{mod} 3).
	\end{cases}
\end{equation*}
Thus $ sn(G \circ K_1) = 2n = n(G) $.\\

Figure \ref{f:L(C3oK1)oK1} gives the uniquely extendable coloring of corona of $ L(C_3 \circ K_1)$ with $ K_1$ where 6 vertices are initially colored, and its unique extension.\\

Figure \ref{f:L(C4oK1)oK1} gives the uniquely extendable coloring of corona of $ L(C_4 \circ K_1)$ with $ K_1$ where 8 vertices are initially colored, and its unique extension.

\begin{figure}[h]
	\centering
	\includegraphics[width=15cm]{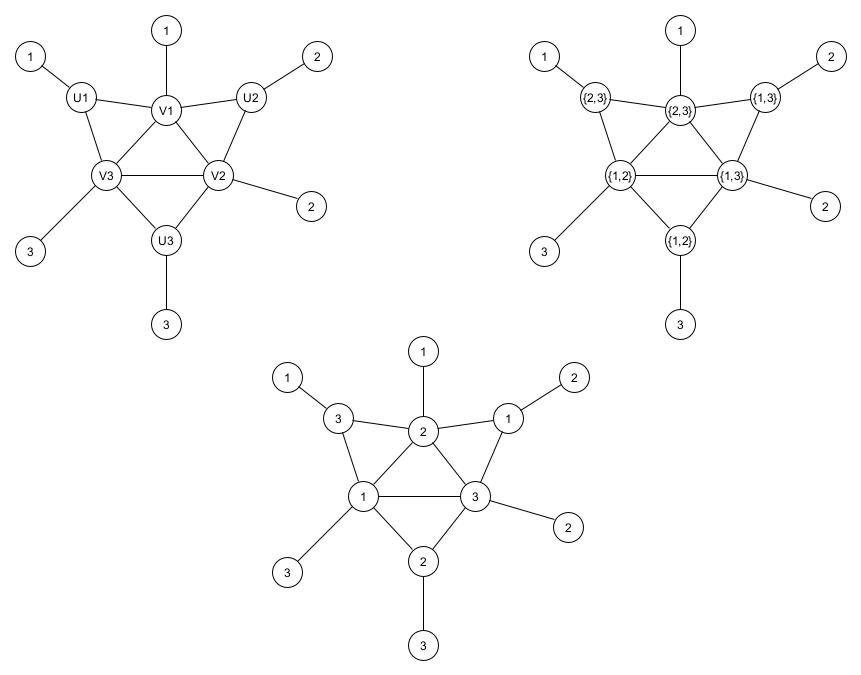}
	\caption{Uniquely extendable coloring of corona of $ L(C_3 \circ K_1)$ and $ K_1$}
	\label{f:L(C3oK1)oK1}
\end{figure}

\begin{figure}[h]
	\centering
	\includegraphics[width=15cm]{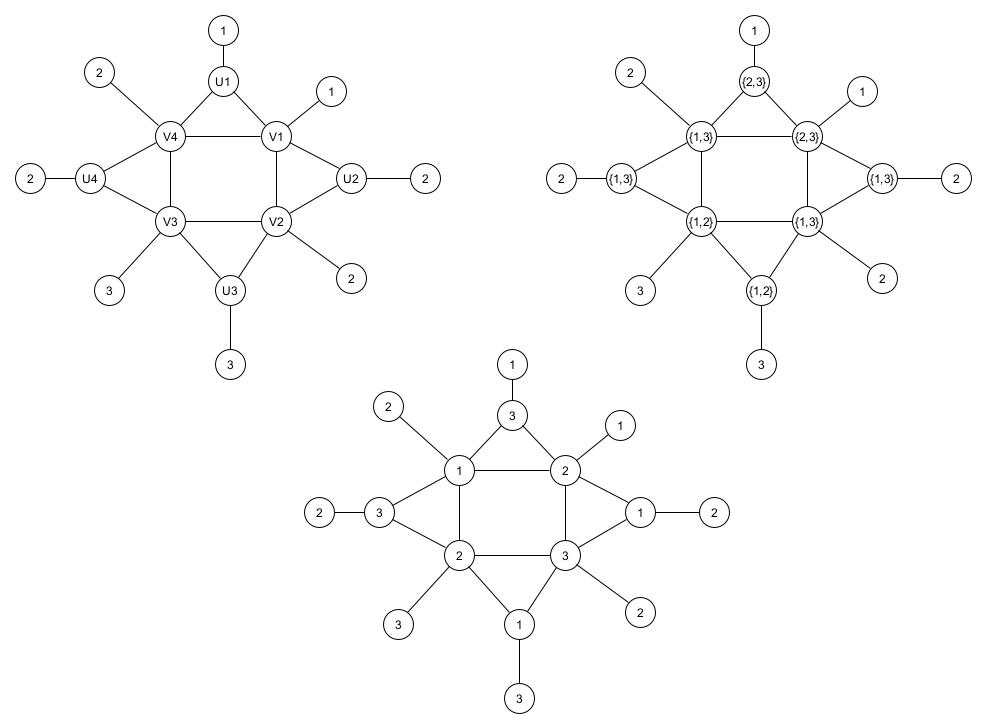}
	\caption{Uniquely extendable coloring of corona of $ L(C_4 \circ K_1)$ and $ K_1$}
	\label{f:L(C4oK1)oK1}
\end{figure}

\begin{thm}
For $n \geq 3$,
\begin{equation*}
  sn(C_n \circ K_1) = \begin{cases} 1; & \mbox{if n is even}\\
n+1;    & \mbox{if n is odd}.
\end{cases}
\end{equation*}
\end{thm}
\begin{proof}
Let $C_n = v_1v_2 \dots v_nv_1$ and let $u_1, u_2, \dots, u_n$ be the pendant vertices attached to $v_1,v_2, \dots, v_n$, respectively. When $n$ is even, $C_n \circ K_1$ is a connected bipartite graph. Hence, $sn(C_n \circ K_1) = 1$.\\

When $n$ is odd, $\chi(C_n \circ K_1) = 3$. By Lemma 1.1, all the $n$ pendent vertices of $C_n \circ K_1$ must be initially colored in any of the Suduku coloring of $C_n \circ K_1$. Therefore, $sn(C_n \circ K_1) \geq n$. If possible assume that $sn(C_n \circ K_1) = n$. That is, all the $n$ pendant vertices alone are initially colored. Then the list of colors available for each vertex $v_i$ in the cycle $C_n$ contains two colors. Therefore, $|L(v_i)| = 2$  and $L(v_i) \subset \{1,2,3\}$. But then the initial coloring is not a Sudoku coloring by Lemma 1.2, a contradiction.\\

Now, let $n = 2k + 1$, where $k$ is not a multiple of 3. Then the following is a uniquely extendable coloring of $C_n \circ K_1$.
\begin{equation*}
  C(v_1) = 3 \and \text{ and } C(u_j) = \begin{cases}
                             1, & \mbox{if } j \equiv 1(\text{mod} 3) \\
                             2, & \mbox{if } j \equiv 2(\text{mod} 3) \\
                             3, & \mbox{if } j \equiv 0(\text{mod} 3).
                           \end{cases}
\end{equation*}

Again, if $n = 2k + 1$, where $k$ is a multiple of 3. Then the following is a uniquely extendable coloring of $C_n \circ K_1$.
\begin{equation*}
  C(v_1) = 3, C(u_n) = 2 \and \text{ and } C(u_j) = \begin{cases}
                             1, & \mbox{if } j \equiv 1(\text{mod} 3) \text{ and } j \neq n \\
                             2, & \mbox{if } j \equiv 2(\text{mod} 3) \\
                             3, & \mbox{if } j \equiv 0(\text{mod} 3).
                           \end{cases}
\end{equation*}
Hence, the theorem.
\end{proof}

Figure~\ref{f:C5} gives the uniquely extendable coloring of $C_5 \circ K_1$ with 6 initially colored vertices and its unique extension.\\

\begin{figure}[h]
	\centering
	\includegraphics[width=15cm]{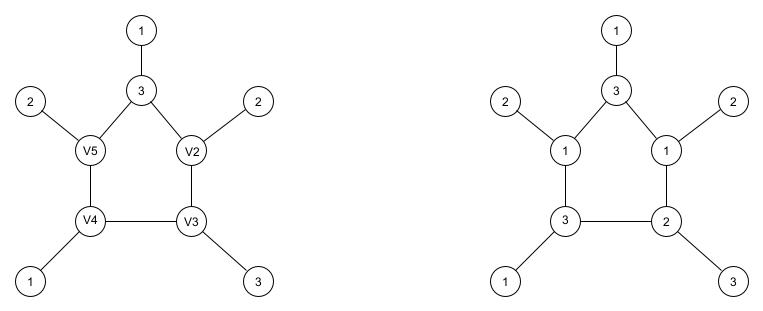}
	\caption{Uniquely extendable coloring of $C_5 \circ K_1$ and its final coloring}
	\label{f:C5}
\end{figure}

\begin{thm}
For $n \geq 3$,
\begin{equation*}
  sn(W_n \circ K_1) = \begin{cases} n+1; & \mbox{if n is even}\\
n+3;    & \mbox{if n is odd}.
\end{cases}
\end{equation*}
\end{thm}
\begin{proof}
Let  the wheel $W_n$ of order $n+1$ be obtained from $C_n = v_1v_2\dots v_nv_1$ by joining a vertex $v$ to every vertex of $C_n$. Let $u_1, u_2, \dots, u_n, u$ be the corresponding pendent vertices in $W_n \circ K_1$.\\

By Lemma 1.1, all the $n+1$ pendent vertices of $W_n \circ K_1$ must be initially colored in any of the Suduku coloring of $W_n \circ K_1$, so that $sn(W_n \circ K_1) \geq n+1$.\\

When $n$ is even, $\chi(W_n \circ K_1) = 3$. Let $C$ be an initial coloring of $W_n \circ K_1$, where $C(u) = C(u_1) = 1$ and $C(u_i) = 2$, for every $i = 2, 3, \dots, n$. By condition $3)$ in observation 1.1, $v$ is a u.c.e vertex which attracts the color 2 alone and in turn $v_1$ becomes a u.c.e vertex which attracts color 3 alone. Now, iteratively, each $v_i$ for $i = 2,3,\dots,n$ becomes u.c.e vertices which attracts colors 1 and 3, alternatively. So $C$ is a uniquely extendable coloring of $W_n \circ K_1$ with $n+1$ initially colored vertices. Hence, $sn(W_n \circ K_1) = n+1$, when $n$ is even.\\

When $n$ is odd, $\chi(W_n \circ K_1) = 4$.\\

Now, let $n = 2k + 1$, where $k$ is not a multiple of 3. Then, using similar arguments, it follows that the following is a uniquely extendable coloring of $W_n \circ K_1$.
\begin{equation*}
  C(v) = 4; C(v_1) = 3; C(u) = 1 \and \text{ and } C(u_j) = \begin{cases}
                             1, & \mbox{if } j \equiv 1(\text{mod} 3) \\
                             2, & \mbox{if } j \equiv 2(\text{mod} 3) \\
                             3, & \mbox{if } j \equiv 0(\text{mod} 3).
                           \end{cases}
\end{equation*}

Again, if $n = 2k + 1$, where $k$ is a multiple of 3. Then the following is a uniquely extendable coloring of $W_n \circ K_1$.
\begin{equation*}
  C(v) = 4; C(v_1) = 3; C(u) = 1; C(u_n) =2, \and C(u_j) = \begin{cases}
                             1, & \mbox{if } j \equiv 1(\text{mod} 3), j \neq n \\
                             2, & \mbox{if } j \equiv 2(\text{mod} 3) \\
                             3, & \mbox{if } j \equiv 0(\text{mod} 3).
                           \end{cases}
\end{equation*}
Therefore, $sn(W_n \circ K_1) \leq n+3$.\\

If possible assume that $sn(W_n \circ K_1) \leq n + 2$ and let $C$ be the corresponding extendable coloring, where $S$ denotes the set of all colored vertices. Since the $n+1$ pendant vertices should be initially colored in any Sudoku coloring, there is at most 1 colored vertex in the wheel $W_n$.\\

If none of the vertices in $W_n$ is colored (that is, $|S| = n+1$), then $G - S$ contains a  cycle $C_n$ with $|L(v_j)| = 3$  and $L(v_j) \subset \{1, 2, 3, 4\}$. Hence $C$ is not a Sudoku coloring, a contradiction. If exactly one vertex in $W_n$ is colored in $C$, then it could be either $v$ or any of the $v_i$'s. If the vertex $v$ is colored with, say color 1, then again $G - S$ contains the cycle $C_n$ with $|L(v_j)| \geq 2$  and $L(v_j) \subseteq \{2, 3, 4\}$. Hence $C$ is not a Sudoku coloring, a contradiction. If $v_i$, without loss of generality, say $v_1$ is colored, then there exists a path $P_{n-1} = v_2v_3\dots v_n$ with $|L(v_j)| \geq 2$. Hence $C$ is not a Sudoku coloring by Lemma 1.3, a contradiction. Hence, $sn(W_n \circ K_1) = n+3$.
\end{proof}

Figure~\ref{f:W4} gives the uniquely extendable coloring of $W_4 \circ K_1$ with 5 initially colored vertices, the color list available for each uncolored vertex and its unique extension.\\

\begin{figure}[h]
	\centering
	\includegraphics[width=15cm]{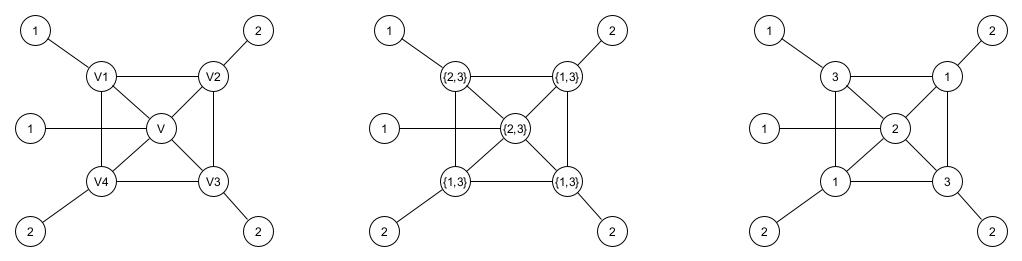}
	\caption{Uniquely extendable coloring of $W_4 \circ K_1$}
	\label{f:W4}
\end{figure}

Figure~\ref{f:W5} gives the uniquely extendable coloring of $W_5 \circ K_1$ with 8 initially colored vertices and its unique extension.\\

\begin{figure}[h]
	\centering
	\includegraphics[width=15cm]{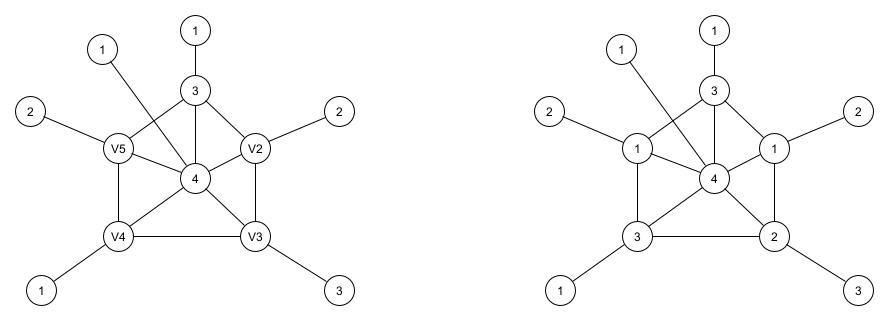}
	\caption{Uniquely extendable coloring of $W_5 \circ K_1$ and its final coloring}
	\label{f:W5}
\end{figure}
Note that when $ n $ is even, $ W_n \circ K_1 $ gives an infinite family of graphs which attains the lower bound of Theorem \ref{th1}.
\begin{thm}
	For $n \geq 4$,
	\begin{equation*}
		sn(W_n \circ K_2) = \begin{cases} n+2; & \mbox{if n is even}\\
			2(n+1);    & \mbox{if n is odd}.
		\end{cases}
	\end{equation*}
\end{thm}
\begin{proof}
	Let  the wheel $W_n$ of order $n+1$ be obtained from $C_n = v_1v_2\dots v_nv_1$ by joining a vertex $v$ to every vertex of $C_n$. Let $u_i$ and $u_i'$ be the vertices of $K_2$ corresponding to the vertex $v_i$ for $i= 1,2, \dots n$ and let $u$ and $u'$ be the vertices of $K_2$ corresponding to the vertex $v$.\\
	
	When $n$ is even, $ \chi(W_n \circ K_2)=3$. By Lemma 1.4, atleast one vertex from each copy of $K_2$ must be initially colored. Therefore, $sn(W_n \circ K_2) \geq n+1$.\\
	Suppose exactly one vertex from each copy of $K_2$ are initially colored. Let $C$ be the corresponding extendable coloring and $S$ be the set of initially colored vertices. Then $G-S$ contains a cycle with $|L(v_i)| = 2$  and $L(v_i) \subset \{1, 2, 3\}$. Then by Lemma 1.2, $C$ is not a Sudoku coloring. Hence $sn(W_n \circ K_2) \geq n+2$.\\
	
	Let $C$ be an initial coloring of $W_n \circ K_2$, where $C(v)=3$, $C(u')=C(u_1')=1$ and $C(u_i')=3$ for $i=2,3, \dots n$. Then $C$ is a uniquely extendable coloring with $n+2$ initially colored vertices. Hence $sn(W_n \circ K_2)= n+2$ when $n$ is even.\\
	
	When $n$ is odd, $ \chi(W_n \circ K_2)=4$. By Observation 1.2, all vertices in each copy of $K_2$ must be initially colored. So, $sn(W_n \circ K_2) \geq 2(n+1)$. Let $C$ be an initial coloring of $W_n \circ K_2$, where $C(u)=C(u_1)=1$, $C(u_i)=2$ for $i=2,3, \dots n$, $C(u')=C(u_i')=4$ for $i= 1,2, \dots (n-1)$ and $C(u_n')=1$. Then $C$ is a uniquely extendable coloring with $2(n+1)$ initially colored vertices. Hence $sn(W_n \circ K_2)= 2(n+1)$ when $n$ is odd.
\end{proof}
Figure~\ref{f:w4} gives the uniquely extendable coloring of $W_4 \circ K_2$ with 6 initially colored vertices and its unique extension.\\

\begin{figure}[h]
	\centering
	\includegraphics[width=13cm]{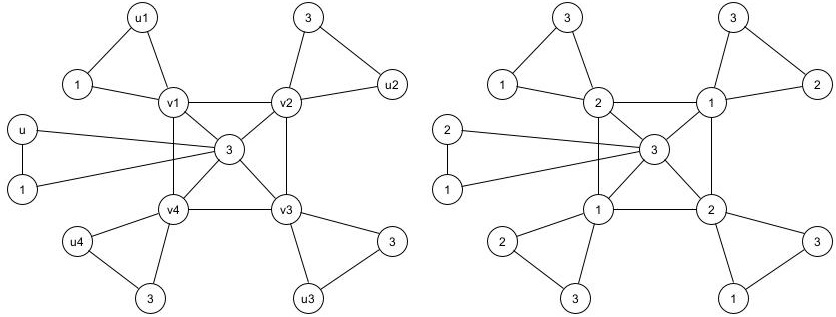}
	\caption{Uniquely extendable coloring of $W_4 \circ K_2$}
	\label{f:w4}
\end{figure}

Figure~\ref{f:w5} gives the uniquely extendable coloring of $W_5 \circ K_2$ with 12 initially colored vertices, the color list available for each uncolored vertex and its unique extension.\\

\begin{figure}[h]
	\centering
	\includegraphics[width=15cm]{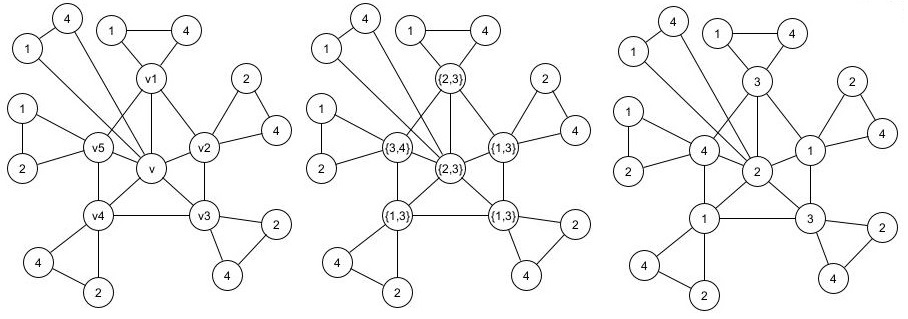}
	\caption{Uniquely extendable coloring of $W_5 \circ K_2$ and its final coloring}
	\label{f:w5}
\end{figure}

\begin{thm}
For $n \geq 3$,
\begin{equation*}
	sn(K_n \circ K_m) = \begin{cases} nm+n-m-1; & \mbox{where $m\leq n-2$ }\\
		n(m-1)+1;    & \mbox{where $m\geq n-1$}.
	\end{cases}
\end{equation*}
\end{thm}

\begin{proof}
Let $v_1,v_2, \dots v_n$ be the vertices of $K_n$. \\

Case I: $m \leq n-2$\\

If possible, assume that there exists an extendable coloring $C$ of $G[S]$ where $S$ is a collection of vertices of $K_n \circ K_m$ with  $|S| \leq nm+n-m-2$. By Observation 1.2, all vertices in each copy of $K_m$ must be initially colored. That is, $nm$ initially colored vertices are from the $n$ copies of $K_m$.Therefore, according to our assumption, atmost $n-m-2$ vertices in $K_n$ are initially colored. That is, atleast $m+2$ vertices in $K_n$ are not yet colored and hence atleast 2 among them are not u.c.e vertices. Therefore, $C$ is not a Sudoku coloring. So, $sn(K_n \circ K_m) \geq nm+n-m-1$.\\
Let $C$ be an initial coloring of $K_n \circ K_m$ where $C(v_i)=i$ for $i=m+2,m+3, \dots n$. Color all vertices of $K_m$ corresponding to $v_i$ using colors other than $i,m+2,m+3, \dots n$. Then $C$ is a uniquely extendable coloring with $nm+n-m-1$ initially colored vertices. Hence $sn(K_n \circ K_m)=nm+n-m-1$.\\

Case II: $m \geq n-1$\\

If possible, assume that there exists an extendable coloring $C$ of $G[S]$ where $S$ is a collection of vertices of $K_n \circ K_m$ with $|S| \leq n(m-1)$. In every copy of $K_m$, atleast $m-1$ vertices must be initially colored, which itself gives all $n(m-1)$ initially colored vertices. Therefore, every vertices of $K_n$ has atleast 2 colors in their color list and hence, atleast 2 among them are not u.c.e vertices. So, $C$ is not a Sudoku coloring and hence, $sn(K_n \circ K_m) \geq n(m-1)+1$.\\
Let $C$ be the following initial coloring of $K_n \circ K_m$.\\
Color all vertices of the copy of $K_m$ corresponding to $v_1$ using colors $2,3, \dots, m+1$. (This forces the color of $v_1$ to be $1$). Color $m-1$ vertices of each copy of $K_m$ corresponding to $v_i$, $i=2,3, \dots, n$ using colors other than $1$ and $i$. (Note that when $1$ and $i$ are removed from the color list $\{1, 2, \dots, m+1\}$ exactly $m-1$ colors are available to color the $m-1$ vertices of $K_m$.) Then $C$ is a uniquely extendable coloring with $n(m-1)+1$ initially colored vertices. Hence $sn(K_n \circ K_m)=n(m-1)+1$.\\

\end{proof}

Figure~\ref{f:K2} gives the uniquely extendable coloring of $K_4 \circ K_2$ with 9 colored vertices and its unique extension.\\
\begin{figure}[h]
	\centering
	\includegraphics[width=15cm]{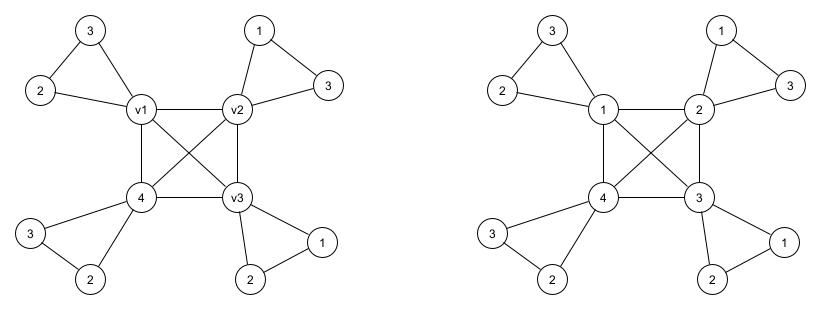}
	\caption{Uniquely extendable coloring of $K_4 \circ K_2$ and its final coloring}
	\label{f:K2}
\end{figure}\\

Figure~\ref{f:K3} gives the uniquely extendable coloring of $K_4 \circ K_3$ with 9 colored vertices and its unique extension.\\
\begin{figure}[h]
	\centering
	\includegraphics[width=15cm]{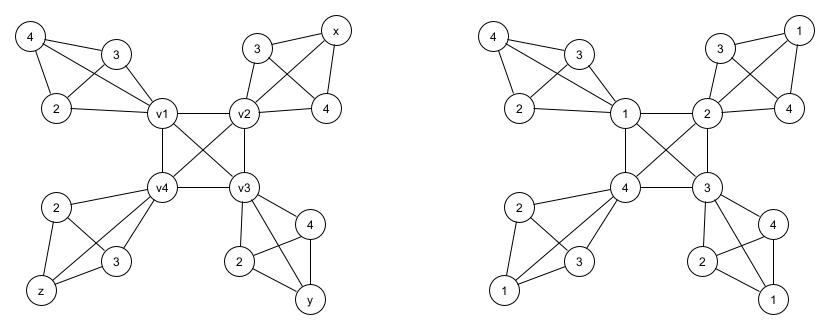}
	\caption{Uniquely extendable coloring of $K_4 \circ K_3$ and its final coloring}
	\label{f:K3}
\end{figure}
\begin{thm}
	For $n \geq 3$, $m \geq 2$, $sn(C_n \circ P_m)= n+1$.
\end{thm}
\begin{proof}
	We know that the chromatic number of $ C_n \circ P_m $ is $3$. Let $ v_1  ,v_2 , \dots , v_n $ be the vertices of $ C_n $ and $ u_{i1} , u_{i2} , \dots , u_{im} $ be the vertices of $ P_m $ corresponding to the vertex $v_i$, $ 1 \leq i \leq n $.	Let $C$ be an initial coloring of $ C_n \circ P_m $. When $n$ is even, $ C(u_{12})=2 $ and $ C(u_{i1})=1 $ , $ 1 \leq i \leq n$. When $n$ is odd, $ C(u_{12})= C(u_{n1})= 2 $ and $ C(u_{i1})=1 $ , $ 1 \leq i \leq n-1$. This is a uniquely extendable coloring with $n+1$ initially colored vertices. Therefore, $sn(C_n \circ P_m) \leq n+1$.\\
	
	At least one vertex of each copy of $P_m$ should be initially colored. Otherwise, there exist at least one path $ P_m = u_{i1}u_{i2} \dots u_{im}$ with $ |L(u_{ij})|= 2$, $1 \leq j \leq m$, $ 1 \leq i \leq n$. Therefore, $ sn(C_n \circ P_m) \geq n$. Now, if $sn(C_n \circ P_m)=n$, that is, one vertex of each copy of $P_m$ is initially colored in $C$, then there exists $C_n = v_1v_2 \dots v_n$ with $|L(v_i)|=2$. Hence $C$ is not a Sudoku coloring by Lemma 1.2, a contradiction. Hence, $sn(C_n \circ P_m)=n+1$.
\end{proof}
Figure~\ref{f:C4} gives the uniquely extendable coloring of $C_4 \circ P_3$ with 5 initially colored vertices and its unique extension.\\

\begin{figure}[h]
	\centering
	\includegraphics[width=13cm]{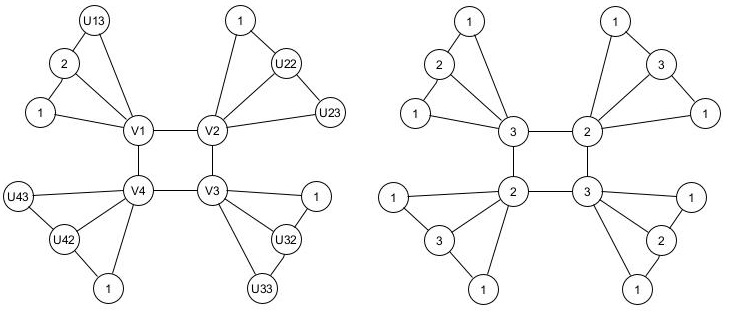}
	\caption{Uniquely extendable coloring of $C_4 \circ P_3$ and its final coloring}
	\label{f:C4}
\end{figure}
Figure~\ref{f:C1} gives the uniquely extendable coloring of $C_5 \circ P_3$ with 6 initially colored vertices and its unique extension.

\begin{figure}[h]
	\centering
	\includegraphics[width=13cm]{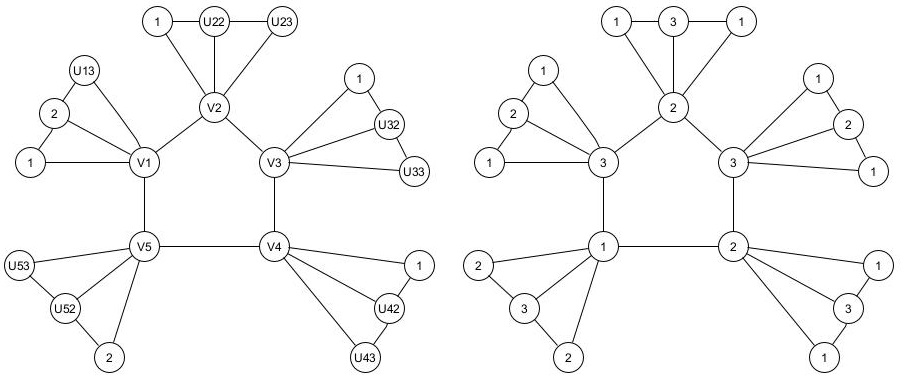}
	\caption{Uniquely extendable coloring of $C_5 \circ P_3$ and its final coloring}
	\label{f:C1}
\end{figure}

\end{document}